\documentclass[12pt,reqno]{amsart}
\usepackage{a4wide}

\numberwithin{equation}{section}

\newtheorem{theorem}{Theorem}[section]
\newtheorem{proposition}[theorem]{Proposition}

\newtheorem{lemma}[theorem]{Lemma}

\theoremstyle{definition}

\newcommand{\la}{\lambda}
\newcommand{\R}{\mathbb{R}}

\begin{document}

\title[  Non-degeneracy of multi-bump  solutions
 ]
{ Non-degeneracy of multi-bump  solutions for  the prescribed scalar curvature equations and applications
 }

 \author{Yuxia Guo, Monica Musso, Shuangjie Peng and Shusen Yan
}

\address{}

\email{}

\address{Department of  Mathematics, Tsinghua University, Beijing, 100084, P.R.China}

\email{yguo@mail.tsinghua.edu.cn   }

\address{Department of Mathematical Sciences University of Bath, Bath BA2 7AY, United
Kingdom}

\email{m.musso@bath.ac.uk}

\address{ School of Mathematics and  Statistics, Central China Normal University, Wuhan, P.R. China}

\email{ sjpeng@mail.ccnu.edu.cn}

\address{ School of Mathematics and  Statistics, Central China Normal University, Wuhan, P.R. China}

\email{ syan@mail.ccnu.edu.cn}

\begin{abstract}

We  consider the following prescribed scalar curvature equations in $\R^N$
\begin{equation}
	\label{eq}
	- \Delta u =K(|y|)u^{2^*-1},\quad u>0 \ \   \mbox{in} \  \R^N, \ \ \
	u \in D^{1, 2}(\R^N),
\end{equation}
where $K(r)$ is a  positive function, $2^*=\frac{2N}{N-2}$.
We modify the methods and improve the results in \cite{GMPY}.

\end{abstract}

\maketitle

\section{Introduction}

It is well known that by using the stereo-graphic projection, the
  prescribed scalar curvature problem on $\mathbb S^N$ can be changed to the following equation:
\begin{equation}\label{1-13-5}
- \Delta u =K(y)u^{2^*-1},\quad u>0 \ \ \   \mbox{in} \  \R^N, \ \
 u \in D^{1, 2}(\R^N).
\end{equation}
Here $2^* = {2N \over N-2}$ and $N\geq 3$.
In the last three decades,  enormous efforts have been devoted to the study of \eqref{1-13-5}.  We refer the
 readers to
 \cite{AAP}--\cite{CL2}, \cite{L1}--\cite{Y} and the references therein.

If  $K(y)$ is radial, infinitely many non-radial solutions are constructed in \cite{WY} for
\begin{equation}
\label{1.4}
 - \Delta u =K(|y|)u^{2^*-1},\quad u>0 \ \   \mbox{in} \  \R^N, \ \ \
 u \in D^{1, 2}(\R^N),
\end{equation}
under the following assumption on $K(r)$:

(K):  There  are  $r_0>0$  and $c_0>0$,
such that
\begin{equation}\label{V}
K(r)= K(r_0) - c_0(r- r_0)^2 +  O(|r-r_0|^3),  \quad r\in (r_0-\delta, r_0+\delta).
\end{equation}
  Without loss of generality, we may assume that $
K(r_0)=1.$
Let us briefly discuss the main results in \cite{WY}.

It is well known ( see \cite{Au,Ta} ) that all solutions to the following problem
\begin{equation}\label{10-18-4}
- \Delta u =u^{2^*-1},\quad u>0 \ \   \mbox{in} \  \R^N, \ \ \
 u \in D^{1, 2}(\R^N)
\end{equation}
are given by
\begin{equation}\label{bub}
U_{x,\mu}(y)= \frac{c_N \mu^{\frac{N-2}2}}{(1+\mu^2 |y-x|^2)^{\frac{N-2}2}}, \quad x\in\mathbb R^N,\; \mu>0,
\end{equation}
where $c_N=\bigl( N (N-2) \bigr)^{\frac{N-2}4}$.
Let $k$ be an integer number and consider the vertices of a regular polygon with $k$ edges in the $(y_1, y_2)$-plane given by
\[
x_j=\Bigl(r \cos\frac{2(j-1)\pi}k, r\sin\frac{2(j-1)\pi}k,0\Bigr),\quad j=1,\cdots,k,
\]
where $0$ denotes the zero vector in $\R^{N-2}$ and   $r\in [r_0-\delta, r_0+\delta]$.
For any point $y \in \R^N$, we set $y=(y',y'')$, $y'\in \R^2$, $y''\in \R^{N-2}$.  Define
\[
\begin{split}
H_s=\Bigl\{ u: \, &  u\;\text{is even in} \;y_h, h=2,\cdots,N,\\
& u(r\cos\theta , r\sin\theta, y'')=
u(r\cos(\theta+\frac{2\pi j}k) , r\sin(\theta+\frac{2\pi j}k), y'')
\Bigr\}
\end{split}
\]
and
\[
W_{r,\mu}(y)=\sum_{j=1}^k U_{x_j, \mu}(y),
\]
where $\mu>0$ is large.
The result obtained in \cite{WY} states the following:

{\bf Theorem~A.}   {\it Suppose that $K(r)$ satisfies (K)  and $N\ge 5$. Then there is an
integer $k_0>0$, such that for any integer  $k\ge k_0$, \eqref{1.4}
has a solution $u_k$ of the form
\[
u_k \approx W_{r_k, \mu_k}(y),
\]
where     $ |r_k-r_0|= O\bigl(\frac1{\mu_k^{1+\sigma}}\bigr) $, $\mu_k\sim k^{\frac{N-2}{N-4}}$.
}

The solutions predicted in Theorem A are obtained by gluing together a very large number of basic profiles \eqref{bub} centered at the vertices of a regular polygon with a large number of edges, and scaled with a parameter $\mu$ that, as $k$ is taken large, diverges to $+\infty$.
 The main term $ W_{r_k, \mu_k}(y)$ of the solution $u_k$ depends on $y''$ radially. To obtain a solution which depends on $y''$ radially,
 we can carry out the reduction procedure in the following space
\[
\begin{split}
 D^{1, 2}(\R^N)\cap \Bigl\{ u: \, &  u\;\text{is even in} \;y_2;\;  u(y', y'')= u(y', |y''|),\\
& u(r\cos\theta , r\sin\theta, y'')=
u(r\cos(\theta+\frac{2\pi j}k) , r\sin(\theta+\frac{2\pi j}k), y'')
\Bigr\},
\end{split}
\]
to ensure that the error term   also  depends on $y''$ radially.

A direct consequence of the proof in \cite{DLY}, together with the estimates in section~2,  is that
the solution satisfying the conditions in Theorem~A is unique.  In particular, such solution
must be radial in $y''$-variable.

Of course, we can also find a solution with $n$-bubbles, whose centers lie  near  the  circle $|y|=r_0$ in the $(y_3, y_4)$-plane.
The question we want to discuss in this paper is whether these two solutions can be glued together to give rise to a new type of solutions.
In other words, we are interested in finding a new solution to \eqref{1-13-5}
whose shape is, at main order,
\begin{equation}\label{86-3-4}
u\approx \sum_{j=1}^k U_{x_j, \mu} +\sum_{j=1}^n U_{p_j, \lambda},
\end{equation}
for $k$ and $n$ large integers,
where
\[
x_j=\Bigl(r \cos\frac{2(j-1)\pi}k, r\sin\frac{2(j-1)\pi}k,0\Bigr),\quad j=1,\cdots,k,
\]
\[
p_j=\Bigl(0,0, t \cos\frac{2(j-1)\pi}n, t\sin\frac{2(j-1)\pi}n,0,\cdots,0\Bigr),\quad j=1,\cdots,n,
\]
and $r$ and $t$ are close to $r_0$. Equation \eqref{1-13-5} is the Euler-Lagrange equation associated to the energy functional
\[
I(u)=\frac12\int_{\R^N} |\nabla u|^2-\frac1{2^*}\int_{\R^N} K(y)|u|^{2^*}.
\]
Thus, roughly speaking, a function of the form \eqref{86-3-4} is an approximate solution to \eqref{1-13-5} provided that the radii $r$, $t$ and the parameters $\mu$ and $\la$ are such that
$$I' ( \sum_{j=1}^k U_{x_j, \mu} +\sum_{j=1}^n U_{p_j, \lambda} ) \sim 0.$$
Having in mind that $\mu , \la \to \infty$, and $r , t \sim r_0$, one easily gets that
\begin{equation}\label{88-3-4}
\begin{split}
&I\Bigl( \sum_{j=1}^k U_{x_j,\mu} +\sum_{j=1}^n U_{p_j,\lambda}  \Bigr)\\
= &(k+n) A +k\Bigl( \frac{B_1}{\mu^2} + \frac{B_2}{\mu^2} (\mu r_0 -r)^2-  \frac{B_3 k^{N-2}}{\mu^{N-2}}\Bigr)+ k  O\Bigl( \frac{B_1}{\mu^{2+\sigma}} + \frac{B_2}{\mu^2} (\mu r_0 -r)^3\Bigr) \\
  &+n\Bigl( \frac{B_1}{\lambda^2} + \frac{B_2}{\lambda^2} (\lambda r_0 -t)^2-  \frac{B_3 n^{N-2}}{\lambda^{N-2}}\Bigr)+ n  O\Bigl( \frac{B_1}{\lambda^{2+\sigma}} + \frac{B_2}{\lambda^2} (\lambda r_0 -t)^3\Bigr),
  \end{split}
\end{equation}
 where $A = \frac12\int_{\mathbb R^N}  |\nabla U_{0, 1}|^2 -\frac1{2^*}\int_{\mathbb R^N} U_{0, 1}^{2^*}$, $B_1$, $B_2$  and $B_3$ are some positive constants,
 and $\sigma>0$ is a small constant.  Observe now that, if  $n>> k$, then the two terms in \eqref{88-3-4} are of different order, which makes it complicated to find a critical point for $I$. Therefore, it is very difficult to use a reduction
 argument to construct solutions of the form \eqref{86-3-4}.
 In fact, this approach has been successfully used in \cite{mmw} (see also \cite{dmpp} and \cite{mw}) to construct finite energy sign-changing solutions in the case $K(y) \equiv 1$, namely
 $$
- \Delta u =|u|^{2^*-2} u , \ \   \mbox{in} \  \R^N,
 u \in D^{1, 2}(\R^N).
$$

 In this paper, we propose an alternative approach and we consider the above problem from completely different point of view.  Recall that our aim is to  glue   $n$-bubbles, whose centers lie in the  circle $|y|=r_0$ in the $(y_3, y_4)$-plane to
 the $k$-bubbling solution $u_k$ described in Theorem A. The linear operator for such problem is
 \[
 Q_n \xi =-\Delta \xi -(2^*-1)K(y)\Bigl( u_k+\sum_{j=1}^n U_{p_j,\lambda}\Bigr)^{2^*-2}\xi.
 \]
 Away from the points $p_j$,  the operator $Q_n$ can be approximated by the
linearized operator around $u_k$, defined by:
 \begin{equation}\label{2-22-12}
 L_k \xi= -\Delta \xi   -(2^*-1)K(y) u_k^{2^*-2}\xi.
 \end{equation}
The new approach we propose is to build the solution with $k$-bubbles in the $(y_1, y_2)$-plane and $n$-bubbles in the $(y_3, y_4)$-plane
as a perturbation of the solution with the $k$-bubbles in the $(y_1, y_2)$-plane. In order to do so, an essential step is to
understand the spectral properties of the liner operator $L_k$ and study its
 invertibility  in some suitable space.

 The main result of this paper is the following.
\begin{theorem}\label{th1-29-3}
Assume $N\ge 5$.
 Suppose that $K(y)$ satisfies $(K)$.
  Let $\xi\in H_s\cap D^{1,2}(\mathbb R^N)$ be a solution of $L_k \xi = 0$.  Then $\xi=0$.

\end{theorem}

A direct consequence of Theorem~\ref{th1-29-3} is the following result for the existence   of new solutions.

 \begin{theorem}\label{th11}

 Suppose that $K(r)$ satisfies the assumptions in Theorem~\ref{th1-29-3} and $N\ge 7$. Let $u_k$ be a solution in Theorem~A and $k>0$ is a large even number. Then there is an
integer $n_0>0$, depending on $k$, such that for any even number  $n\ge n_0$,  \eqref{1.4}
has a solution of the form \eqref{86-3-4} for some $t_n\to r_0$  and $\lambda_n\sim n^{\frac{N-2}{N-4}}$  .

 \end{theorem}

Local uniqueness of single bubbling solutions for elliptic problems with critical growth was first studied in \cite{G} by using a degree-counting method,
while in \cite{DLY,GPY}, the authors used the local Pohozaev identities to deal with the local uniqueness problem for multi-bubbling solutions. The use of
the local Pohozaev identities not only simplifies  the estimates, but it also makes it possible to study the local uniqueness of solutions with large numbers
of bubbles. It is well known that the non-degeneracy of the solution and the uniqueness of such solution are two very closely related  problems. In this paper,
we shall show that the local Pohozaev identities also play an important role in the study of the non-degeneracy of the multi-bubbling solutions.

This paper is organized as follows. In section~2, we shall prove the main theorem by using the local Pohozaev identities.  As an application of this main result,
new solutions for \eqref{1.4} are constructed in section~3.

\section{The non-degeneracy of the solutions}

 For the simplicity of the notations, we consider the following equivalent problem

 \begin{equation}\label{1-31-5-21}
  -\Delta u=K_k(|y|)u^{2^*-1},
 \end{equation}
 where $K_k(y)=K\bigl(\frac y k\bigr)$.

We aim to find a $k$-bubbling solution $\bar u_k$ for \eqref{1-31-5-21}, satisfying

\[
\bar u_k = W_{\bar r_k, \bar \mu_k}(y)+\bar \omega_k,
\]
where $\bar x_{k, j}=  (\bar r_k \cos\frac{2(j-1) \pi}{k},\bar r_k\sin\frac{2(j-1)\pi}k, 0)$, $\bar \mu_k>0$ is large, and
$\bar \omega_k$ is a small perturbation term.  For this purpose, we
 define the norm as
\[
\| u\|_*= \sup_{y\in \mathbb R^N} |u(y)|\Bigl(\sum_{j=1}^k \frac{\bar \mu_k^{\frac{N-2}2}}{ ( 1+ \bar \mu_k |y-\bar x_{k, j}|)^{\frac{N-2}2+\tau}}\Bigr)^{-1}.
\]
where    $\tau=\frac{N-4}{N-2}+\bar \sigma$ with $\bar\sigma>0$ small.
 In \cite{WY},  the following result is proved.

{\bf Theorem~B.}   {\it Suppose that $K(r)$ satisfies (K)  and $N\ge 5$. Then there is an
integer $k_0>0$, such that for any integer  $k\ge k_0$,
\eqref{1-31-5-21} has a solution $\bar u_k$ of the form
\[
\bar u_k = W_{\bar r_k, \bar \mu_k}(y)+\bar \omega_k,
\]
where $\bar \omega_k\in H_s\cap  D^{1, 2}(\R^N),$
and as $k\to +\infty$,
\[
 |k^{-1}\bar r_k-r_0|= O\bigl(\frac1{(k \bar \mu_k)^{1+\sigma}}\bigr) ,\quad \| \bar \omega_k \|_* =  O\bigl(\frac1{(k \bar \mu_k)^{1+\sigma}}\bigr)
\]
 and

 \begin{equation}\label{2-31-5-21}
 \begin{split}
 & -\frac1{2^*} \frac{\Delta K(k^{-1}\bar x_{k, 1})}{N\bar \mu_k^3 k^2}\int_{\mathbb R^N} U^{2^*} |y|^2\\
 =&
  \frac{N-2}{2\bar \mu^{N-1}_k}\sum_{j=1}^k
 \frac{c_N}{|\bar x_{k,j}-\bar x_{k, 1}|^{N-2}}\int_{\mathbb R^N} U^{2^*-1}
  +O\bigl(\frac1{\bar \mu_k (k\bar \mu_k)^{2+\sigma}}\bigr),
  \end{split}
 \end{equation}
 for some $\sigma>0$.

}

For the estimate of $\bar u_k$, we have

\begin{lemma}\label{l1-26-4}
There exists a constant $C>0$ such that
\[
|\bar u_k(y)|\le  C \sum_{j=1}^k \frac{\bar \mu_k^{\frac{N-2}2}}{ ( 1+ \bar \mu_k |y-\bar  x_{k,j}|)^{N-2}}, \quad {\mbox {for all}} \quad y \in \R^N.
\]

\end{lemma}

\begin{proof}

Let  $\tilde u_k (y) =  \bar  \mu_k^{-\frac{N-2}2} \bar u_k (\bar  \mu_k^{-1} y) $.  Then
\[
-\Delta \tilde u_k =  K_k (\bar  \mu_k^{-1} y) \tilde u^{2^*-1}_k.
\]
 We have
 \[
\tilde  u_k (y)= \int_{\mathbb R^N} \frac1{|z-y|^{N-2}}   K (\bar \mu_k^{-1} z) \tilde u^{2^*-1}_k(z)\,dz.
 \]
By Theorem~B,  we find
\[
\begin{split}
|\tilde u_k (y)|\le & C\int_{\mathbb R^N}  \frac1{|z-y|^{N-2}} \tilde  u_k^{2^*-1}
\,dz\\
\le & C \int_{\mathbb R^N}  \frac1{|z-y|^{N-2}}\Bigl(  \sum_{j=1}^k \frac{1}{ ( 1+  |z-\tilde x_{k, j}|)^{\frac{N-2}2+\tau}}\Bigr)^{2^*-1}\\
\le & C\int_{\mathbb R^N}  \frac1{|z-y|^{N-2}}  \sum_{j=1}^k \frac{1}{ ( 1+  |z-\tilde x_{k, j}|)^{\frac{N+2}2+\tau_1 +\frac{(N+2)(\tau-\tau_1)}{N-2}}}
\Bigl( \sum_{j=1}^k \frac{1}{  |\tilde x_{k, 1}-\tilde x_{k, j}|^{\tau_1}}\Bigr)^{\frac4{N-2}}\\
\le & C \sum_{j=1}^k \frac{1}{ ( 1+  |y-\tilde x_{k, j}|)^{\frac{N-2}2+\tau_1 +\frac{(N+2)(\tau-\tau_1)}{N-2}}},
\end{split}
\]
where $\tilde x_{k, j}=\bar  \mu_k \bar  x_{k, j}$,  and  $\tau_1\in (\frac{N-4}{N-2}, \tau)$.  Noting that
\[
\frac{N-2}2+\tau_1 +\frac{(N+2)(\tau-\tau_1)}{N-2}=\frac{N-2}2+\tau +\frac{4(\tau-\tau_1)}{N-2}> \frac{N-2}2+\tau,
\]
we can continue this process to prove the result.

\end{proof}

 Let $u\in D^{1,2}(\mathbb R^N)$ and $\xi\in D^{1,2}(\mathbb R^N)$ satisfy

 \begin{equation}\label{1-26-11}
  -\Delta u=K_k(|y|)u^{2^*-1},
 \end{equation}
 and
 \begin{equation}\label{2-26-11}
  -\Delta \xi  =(2^*-1)  K_k(|y|)  u^{2^*-2}\xi.
 \end{equation}
 Then standard arguments give $u,\, \xi\in L^\infty(\mathbb R^N)$ and

 \[
 |u(y)|,\; \; |\xi(y)|\le \frac C{(1+ |y|)^{N-2}}.
 \]

Assume that $\Omega $ is a smooth domain in $\R^N$.
 We have the following identities.

 \begin{lemma}\label{l0-14-12}
  It holds

\begin{equation}\label{1-20-12}
  \begin{split}
 &- \int_{\partial \Omega}\frac{\partial u}{\partial \nu} \frac{\partial \xi}{\partial y_i}- \int_{\partial \Omega}\frac{\partial \xi}{\partial \nu} \frac{\partial u}{\partial y_i}
 +
 \int_{\partial \Omega} \bigl\langle \nabla u, \nabla \xi\bigr\rangle \nu_i  -\int_{\partial \Omega} K_k(|y|)u^{2^*-1} \xi \nu_i\\
 = &-\int_\Omega u^{2^*-1} \xi \frac{\partial K_k(|y|) }{\partial y_i},
 \end{split}
 \end{equation}
and
\begin{equation}\label{1-26-2}
  \begin{split}
 &\int_\Omega u^{2^*-1} \xi  \langle \nabla K_k(y), y-x_0\rangle\\
 = &\int_{\partial \Omega } K_k(|y|)   u^{2^*-1} \xi  \langle \nu, y-x_0\rangle\\
 &+ \int_{\partial \Omega}\frac{\partial u}{\partial \nu} \langle \nabla \xi, y-x_0\rangle
 + \int_{\partial \Omega}\frac{\partial \xi}{\partial \nu} \langle \nabla u, y-x_0\rangle- \int_{\partial \Omega } \bigl\langle  \nabla
 u, \nabla \xi\bigr\rangle \langle \nu, y-x_0\rangle\\
 &
  +\frac{N-2}2  \int_{\partial \Omega} \xi \frac{\partial u}{\partial \nu}+\frac{N-2}2 \int_{\partial \Omega} u \frac{\partial \xi}{\partial \nu}.
 \end{split}
 \end{equation}

 \end{lemma}

 \begin{proof}

{\it Proof of \eqref{1-20-12}}.  We have\ \
\begin{equation}\label{3-26-11}
 \begin{split}
 & \int_\Omega\Bigl( -\Delta u \frac{\partial \xi}{\partial y_i} + (-\Delta \xi) \frac{\partial u }{\partial y_i}\Bigr)\\
 =&\int_\Omega K_k(|y|)\Bigl(  u^{2^*-1}  \frac{\partial \xi}{\partial y_i} + (2^*-1) u^{2^*-2} \xi \frac{\partial u }{\partial y_i}\Bigr).
 \end{split}
 \end{equation}

 It is easy to check that
\begin{equation}\label{4-26-11}
  \begin{split}
& \int_\Omega K_k(|y|)\Bigl(  u^{2^*-1}  \frac{\partial \xi}{\partial y_i} + (2^*-1) u^{2^*-2} \xi \frac{\partial u }{\partial y_i}\Bigr)\\
 = &\int_\Omega K_k(|y|) \frac{\partial ( u^{2^*-1} \xi) }{\partial y_i}=-\int_\Omega u^{2^*-1} \xi \frac{\partial K_k(|y|) }{\partial y_i}
 +\int_{\partial \Omega} K_k(|y|)u^{2^*-1} \xi \nu_i.
 \end{split}
 \end{equation}
Moreover,
\begin{equation}\label{6-26-11}
 \begin{split}
 & \int_\Omega\Bigl( -\Delta u \frac{\partial \xi}{\partial y_i} + (-\Delta \xi) \frac{\partial u }{\partial y_i}\Bigr)\\
 =& - \int_{\partial \Omega}\frac{\partial u}{\partial \nu} \frac{\partial \xi}{\partial y_i} + \int_\Omega \frac{\partial u}{\partial y_j} \frac{\partial^2 \xi}{\partial y_i\partial
 y_j}
 - \int_{\partial \Omega}\frac{\partial \xi}{\partial \nu} \frac{\partial u}{\partial y_i} + \int_\Omega \frac{\partial \xi}{\partial y_j} \frac{\partial^2 u}{\partial y_i\partial
 y_j}\\
 =& - \int_{\partial \Omega}\frac{\partial u}{\partial \nu} \frac{\partial \xi}{\partial y_i}- \int_{\partial \Omega}\frac{\partial \xi}{\partial \nu} \frac{\partial u}{\partial y_i}
 +\int_\Omega \frac{\partial }{\partial y_i} \Bigl(  \frac{\partial u}{\partial y_j} \frac{\partial \xi}{\partial y_j} \Bigr)\\
 =& - \int_{\partial \Omega}\frac{\partial u}{\partial \nu} \frac{\partial \xi}{\partial y_i}- \int_{\partial \Omega}\frac{\partial \xi}{\partial \nu} \frac{\partial u}{\partial y_i}+
 \int_{\partial \Omega} \bigl\langle \nabla u, \nabla \xi\bigr\rangle \nu_i.
 \end{split}
 \end{equation}
 So we have proved \eqref{1-20-12}.

\medskip
{\it Proof of \eqref{1-26-2}}. \  We have
\begin{equation}\label{2-26-2}
 \begin{split}
 & \int_\Omega\Bigl( -\Delta u \langle \nabla \xi, y-x_0\rangle  + (-\Delta \xi)  \langle \nabla u, y-x_0\rangle\Bigr)\\
 =&\int_\Omega K_k(|y|)\Bigl(  u^{2^*-1}  \langle \nabla \xi, y-x_0\rangle + (2^*-1) u^{2^*-2} \xi \langle \nabla u, y-x_0\rangle\Bigr).
 \end{split}
 \end{equation}

 It is easy to see that
 \begin{equation}\label{3-26-2}
 \begin{split}
 &\int_\Omega K_k(|y|)\Bigl(  u^{2^*-1}  \langle \nabla \xi, y-x_0\rangle + (2^*-1) u^{2^*-2} \xi \langle \nabla u, y-x_0\rangle\Bigr)\\
 =& \int_\Omega K_k(|y|)    \langle \nabla ( u^{2^*-1} \xi), y-x_0\rangle\\
 =&  \int_{\partial \Omega } K_k(|y|)   u^{2^*-1} \xi  \langle \nu, y-x_0\rangle
  -\int_\Omega u^{2^*-1} \xi  \langle \nabla K_k(y), y-x_0\rangle -N \int_\Omega K_k(|y|)     u^{2^*-1} \xi,
 \end{split}
 \end{equation}
 where $\nu$ is the outward unit normal of $\partial \Omega$  at $y\in\partial\Omega$.
 Moreover,
\begin{equation}\label{4-26-2}
 \begin{split}
 & \int_\Omega\Bigl( -\Delta u \langle \nabla \xi, y-x_0\rangle  + (-\Delta \xi)  \langle \nabla u, y-x_0\rangle\Bigr)\\
 =& - \int_{\partial \Omega}\frac{\partial u}{\partial \nu} \langle \nabla \xi, y-x_0\rangle+ \int_\Omega \frac{\partial u}{\partial y_j}\bigl\langle  \nabla \frac{\partial \xi}{ \partial
 y_j}, y-x_0\bigr\rangle + \int_\Omega \bigl\langle  \nabla
 u, \nabla \xi\bigr\rangle\\
 &
 - \int_{\partial \Omega}\frac{\partial \xi}{\partial \nu} \langle \nabla u, y-x_0\rangle + \int_\Omega \frac{\partial \xi}{\partial y_j}\bigl\langle  \nabla \frac{\partial u}{ \partial
 y_j}, y-x_0\bigr\rangle + \int_\Omega \bigl\langle  \nabla
 u, \nabla \xi\bigr\rangle\\
 =& - \int_{\partial \Omega}\frac{\partial u}{\partial \nu} \langle \nabla \xi, y-x_0\rangle
 - \int_{\partial \Omega}\frac{\partial \xi}{\partial \nu} \langle \nabla u, y-x_0\rangle+ \int_{\partial \Omega } \bigl\langle  \nabla
 u, \nabla \xi\bigr\rangle \langle \nu, y-x_0\rangle\\
 &+(2-N)  \int_\Omega \bigl\langle  \nabla
 u, \nabla \xi\bigr\rangle.
 \end{split}
 \end{equation}

 We also have
\begin{equation}\label{5-26-2}
 \begin{split}
 2^*  \int_\Omega  K_k(|y|) u^{2^*-1} \xi= & \int_\Omega\bigl( -\xi \Delta u   + u (-\Delta \xi) \bigr)\\
 =&2\int_\Omega \bigl\langle  \nabla
 u, \nabla \xi\bigr\rangle - \int_{\partial \Omega} \xi \frac{\partial u}{\partial \nu}- \int_{\partial \Omega} u \frac{\partial \xi}{\partial \nu},
 \end{split}
 \end{equation}
 which gives
\begin{equation}\label{6-26-2}
 \begin{split}
 \int_\Omega \bigl\langle  \nabla
 u, \nabla \xi\bigr\rangle = &\frac{ 2^* }2  \int_\Omega  K(|y|) u^{2^*-1} \xi
 +\frac12  \int_{\partial \Omega} \xi \frac{\partial u}{\partial \nu}+\frac12 \int_{\partial \Omega} u \frac{\partial \xi}{\partial \nu}.
 \end{split}
 \end{equation}
 Thus, the result follows.

\end{proof}

We define the linear operator
\begin{equation}\label{2-27-11}
\bar  L_k \xi= -\Delta \xi -(2^*-1)  K_k(|y|)\bar u_k^{2^*-2}\xi.
 \end{equation}








We now  prove Theorem~\ref{th1-29-3}, arguing  by contradiction.   Suppose that there are $k_m\to +\infty$,  $\xi_{k_m}\in H_s\cap D^{1,2}(\mathbb R^N)$, satisfying
 $\|\xi_{k_m}\|_*=1$, and
\begin{equation}\label{10-13-12}
\bar  L_{k_m}\xi_{k_m} =0.
 \end{equation}

Let
\begin{equation}\label{1-10-4}
\tilde \xi_{m} (y)=\bar \mu_{k_m}^{-\frac{N-2}2} \xi_{k_m}
(\bar \mu_{k_m}^{-1}y+ \bar  x_{k_m, 1}).
\end{equation}

\begin{lemma}\label{l1-10-4}

 It holds
\begin{equation}\label{10-9-4}
\tilde \xi_m \to  b_0\psi_0+  b_1\psi_1,
\end{equation}
uniformly in $C^1(B_R(0))$ for any $R>0$, where $b_0$   and $b_1$ are  some constants,
\[
\psi_0=\frac{\partial U_{0, \mu}}{\partial \mu}\Bigr|_{\mu=1},\quad \psi_i= \frac{\partial U_{0,1}}{\partial y_i}, \ \  i=1, ..., N.
\]

\end{lemma}

\begin{proof}

  In view of
$|\tilde \xi_{m}|\le C$, we may assume that $\tilde \xi_{m}\to \xi$ in
$ C_{loc}(\mathbb R^N)$.   Then $\xi$ satisfies
\begin{equation}\label{3-8-4}
-\Delta \xi = (2^*-1) U^{2^*-2}\xi, \quad \text{in}\; \mathbb R^N,
\end{equation}
which gives
\begin{equation}\label{4-8-4}
\xi =\sum_{i=0}^{N} b_{i}\psi_i.
\end{equation}

Since $\xi_m$ is even in $y_i$, $i=2,\cdots, N$,  it holds $b_i=0$, $i=2,\cdots,N$.

\end{proof}

Similar to Lemma~\ref{l1-26-4}, we can prove the following result.

\begin{lemma}\label{l10-28-2}

There exists a constant $C>0$ such that
\[
|\xi_{k_m}(y)|\le  C \sum_{j=1}^{k_m} \frac{\bar \mu_{k_m}^{\frac{N-2}2}}{ ( 1+ \bar \mu_{k_m} |y-\bar x_{k_m,j}|)^{N-2}}, \quad {\mbox {for all}} \quad y \in \R^N.
\]

\end{lemma}

\begin{proof}
 By  Lemma~\ref{l1-26-4}, we have

\[
\begin{split}
|\tilde \xi _m (y)|\le & C\int_{\mathbb R^N}  \frac1{|z-y|^{N-2}} \Bigl(  \sum_{j=1}^{k_m}\frac{1}{ ( 1+  |z-\tilde x_{k_m, j}|)^{N-2}}\Bigr)^{\frac4{N-2}} |\tilde \xi_m|
\,dz\\
\le & C\|\tilde \xi_m\|_*\int_{\mathbb R^N}  \frac1{|z-y|^{N-2}}\Bigl(  \sum_{j=1}^{k_m} \frac{1}{ ( 1+  |z-\tilde x_{k_m, j}|)^{\frac{N-2}2+\tau}}\Bigr)^{2^*-1}\\
\le & C \sum_{j=1}^{k_m}\frac{1}{ ( 1+  |y-\tilde x_{k_m, j}|)^{\frac{N-2}2+\tau_1 +\frac{(N+2)(\tau-\tau_1)}{N-2}}}.
\end{split}
\]
We can continue this process to prove the result.
\end{proof}

We will use \eqref{1-20-12}
for $u=\bar u_k$ and  $\Omega = B_\delta(\bar x_{k, 1})$.
Let  $G(y, x)=\frac1{(N-2)\omega_{N-1}}\frac1{|y-x|^{N-2}}$, where $\omega_{N-1}$ is the area of the unit sphere.

\begin{lemma}\label{l1-31-5-21}
For $\delta>0$ small, we have

\[
\bar u_{k_m}(y) =\frac{A_m}{\bar \mu_{k_m}^{\frac{N-2}2}} \sum_{j=1}^{k_m} G(y, \bar x_{k_m, j})
+ O\bigl(\frac1{\bar \mu_{k_m}^{\frac{N-2}2+\sigma}}\bigr),\quad \text{in}\;  \cup_{j=1}^k (B_{2\delta}(\bar x_{k_m, j})\setminus (B_{\frac12\delta}(\bar x_{k_m, j}),
\]

\[
\partial \bar u_{k_m}(y) =\frac{A_m}{\bar \mu_k^{\frac{N-2}2}} \sum_{j=1}^{k_m} \partial_x G(y, \bar x_{k_m, j})
+ O\bigl(\frac1{\bar \mu_{k_m}^{\frac{N-2}2+\sigma}}\bigr),\quad \text{in}\;  \cup_{j=1}^k (B_{2\delta}(\bar x_{k_m, j})\setminus (B_{\frac12\delta}(\bar x_{k_m, j}),
\]

\[
\xi_{k_m}(y) =\frac{B_m}{\bar \mu_k^{\frac{N-2}2}} \sum_{j=1}^{k_m} G(y, \bar x_{k_m, j})
+ O\bigl(\frac1{\bar \mu_{k_m}^{\frac{N-2}2+\sigma}}\bigr),\quad \text{in}\;  \cup_{j=1}^k (B_{2\delta}(\bar x_{k_m, j})\setminus (B_{\frac12\delta}(\bar x_{k_m, j}),
\]
and

\[
\partial \xi_{k_m}(y)=\frac{B_m}{\bar \mu_k^{\frac{N-2}2}} \sum_{j=1}^{k_m}\partial_x G(y, \bar x_{k_m, j})
+ O\bigl(\frac1{\bar \mu_{k_m}^{\frac{N-2}2+\sigma}}\bigr),\quad \text{in}\;  \cup_{j=1}^k (B_{2\delta}(\bar x_{k_m, j})\setminus (B_{\frac12\delta}(\bar x_{k_m, j}),
\]
where

\[
A_m= \int_{B_\delta(\bar x_{k_m, j})} K_{k_m}(y) \bar u^{2^*-1}_{k_m}\to \int_{\mathbb R^N} U^{2^*-1},
\]

 \[
 B_m  =(2^*-1)\int_{B_\delta(\bar x_{k_m, j})} K_{k_m}(y) \bar u^{2^*-2}_{k_m}\xi_{k_m}\to (2^*-1)b_0 \int_{\mathbb R^N} U^{2^*-2}\psi_0
 \]
   and $\sigma>0$ is a constant.
\end{lemma}

\begin{proof}

 For simplicity, we drop the subscript $m$ in $k_m$. We have

\[
\bar  u_k(x)=\int_{\mathbb R^N} G(y, x)K_k(y) \bar u^{2^*-1}_k,\quad \partial \bar u_k(x)=\int_{\mathbb R^N} \partial_x G(y, x)K_k(y) \bar u^{2^*-1}_k.
\]
Now we do the estimate for the derivatives.

\[
\begin{split}
\partial_{x_j} \bar  u_k(x)=&\sum_{j=1}^k \int_{B_\delta(\bar x_{k, j})} \partial_{x_j} G(y, x)K_k(y) \bar  u^{2^*-1}_k +
 \sum_{j=1}^k \int_{\Omega_j\setminus B_\delta(\bar x_{k, j})} \partial_{x_j} G(y, x)K_k(y) \bar  u^{2^*-1}_k
\end{split}
\]
We assume that $x\in \Omega_1$.  Then for $y\in \Omega_j$, $j\ge 2$, it holds

\[
|y-x|\ge |y-\bar x_1|-2\delta\ge \frac12 |\bar x_j-\bar x_1| -2\delta \ge\frac14|\bar x_j-\bar x_1|.
\]
Thus,

\[
\begin{split}
&\Bigl|\sum_{j=2}^k \int_{\Omega_j\setminus B_\delta(\bar x_{k, j})} \partial_{x_j} G(y, x)K_k(y) u^{2^*-1}_k\Bigr|\le
\sum_{j=2}^k\frac1{|\bar x_j-\bar x_1|^{N-1}}\int_{\Omega_j\setminus B_\delta(\bar x_{k, j})}  u^{2^*-1}_k
\\
\le & \frac C{\bar \mu_k^{\frac{N-2}2+\sigma}}\sum_{j=2}^k\frac1{|\bar x_j-\bar x_1|^{N-1}}\le \frac C{\bar \mu_k^{\frac{N-2}2+\sigma}}.
\end{split}
\]
On the other hand, we have

\[
\int_{\Omega_1\setminus B_\delta(\bar x_{k, j})} \partial_{x_j} G(y, x)K_k(y) u^{2^*-1}_k=O\bigl( \frac 1{\bar \mu_k^{\frac{N-2}2+\sigma}}\bigr),
\]
and

\[
\int_{B_\delta(\bar x_{k, j})} \partial_{x_j} G(y, x)K_k(y) u^{2^*-1}_k=\partial_{x_j} G(y, \bar x_{k, j})) \int_{B_\delta(\bar x_{k, j})} K_k(y) u^{2^*-1}_k+O\bigl( \frac 1{\bar \mu_k^{\frac{N-2}2+\sigma}}\bigr).
\]
Using Lemma~\ref{l1-26-4}, it is easy to prove

\[
 \int_{B_\delta(\bar x_{k_m, j})} K_{k_m}(y) \bar u^{2^*-1}_{k_m}\to \int_{\mathbb R^N} U^{2^*-1}.
\]
So, we obtain the estimate for $\nabla u_k$.

 Similarly, we have

\[
\xi_k(x)=\int_{\mathbb R^N} G(y, x) (2^*-1)u^{2^*-2}_k\xi_k,\quad \partial_x \xi_k(x)=\int_{\mathbb R^N} \partial_x G(y, x)(2^*-1) u^{2^*-2}_k\xi_k.
\]
Using Lemmas~\ref{l10-28-2} and \ref{l1-10-4}, we can prove

\[
(2^*-1)\int_{B_\delta(\bar x_{k_m, j})} K_{k_m}(y) \bar u^{2^*-2}_{k_m}\xi_{k_m}\to (2^*-1)b_0 \int_{\mathbb R^N} U^{2^*-2}\psi_0.
 \]
So we can obtain the estimate for $\xi_k$.
\end{proof}

Now we are ready to prove the following result.

\begin{lemma}\label{l1-14-12}

 It holds $b_0=b_1=0$.

\end{lemma}

\begin{proof}

\textbf{Step~1.}  We apply the identities in Lemma~\ref{l0-14-12}  in the domain  $B_\delta(\bar x_{k, 1})$:
\begin{equation}\label{2-14-12}
   \begin{split}
 & - \int_{\partial B_\delta(\bar x_{k_m, 1}) }\frac{\partial u_{k_m}}{\partial \nu} \frac{\partial \xi_{k_m}}{\partial y_1}- \int_{\partial B_\delta(\bar x_{k_m, 1})}\frac{\partial \xi_{k_m}}{\partial \nu} \frac{\partial u_{k_m}}{\partial y_1}
 +
 \int_{\partial B_\delta(\bar x_{k_m, 1})} \bigl\langle \nabla u_{k_m}, \nabla \xi_{k_m}\bigr\rangle \nu_1\\
  &-\int_{\partial B_\delta(\bar x_{k_m, 1})} K_{k_m}(|y|)u_{k_m}^{2^*-1} \xi_m \nu_1
 =-\int_{  B_\delta(\bar x_{k_m, 1})} u_{k_m}^{2^*-1} \xi_{k_m} \frac{\partial K_{k_m}(|y|) }{\partial y_1}.
 \end{split}
 \end{equation}

 It is easy to check by using Lemma~\ref{l1-31-5-21} that

\begin{equation}\label{10-31-5-21}
  \int_{\partial B_\delta(\bar x_{k_m, 1})} K_{k_m}(|y|)u_{k_m}^{2^*-1} \xi_{k_m} \nu_1
 =O\bigl( \frac 1{(k_m\bar \mu_{k_m})^{2+\sigma}}\bigr).
 \end{equation}

Let
\begin{equation}\label{2-21-12}
  Q(u, \xi, \delta)=- \int_{\partial B_\delta(\bar x_{k_m, 1}) }\frac{\partial u}{\partial \nu} \frac{\partial \xi}{\partial y_1}- \int_{\partial B_\delta(\bar x_{k_m, 1})}\frac{\partial \xi}{\partial \nu} \frac{\partial u}{\partial y_1}
 +
 \int_{\partial B_\delta(\bar x_{k_m, 1})} \bigl\langle \nabla u, \nabla \xi\bigr\rangle \nu_1.
 \end{equation}

Using Lemma~\ref{l1-31-5-21} , we see that

\begin{equation}\label{11-31-5-21}
   Q(u_{k_m}, \xi_{k_m}, \delta)
 =Q\bigl( \frac{A_m}{\bar \mu_{k_m}^{\frac{N-2}2}} \sum_{j=1}^{k_m} G(y, \bar x_{k_m, j}), \frac{ B_m}{\bar \mu_{k_m}^{\frac{N-2}2}} \sum_{j=1}^{k_m} G(y, \bar x_{k_m, j}) , \delta  \bigr)+O\bigl( \frac 1{(k_m\bar \mu_{k_m})^{2+\sigma}}\bigr).
 \end{equation}

 Since $\sum_{j=1}^{k_m} G(y, \bar x_{k_m, j})$ are harmonic in $B_\delta(\bar x_{k_m, 1})\setminus B_\theta(\bar x_{k_m, 1})$ for any $\theta\in (0, \delta)$, it holds

 \begin{equation}\label{12-31-5-21}
 \begin{split}
 &  Q\bigl( \sum_{j=1}^{k_m} G(y, \bar x_{k_m, j}),  \sum_{j=1}^{k_m} G(y, \bar x_{k_m, j}) , \delta  \bigr)=
   Q\bigl( \sum_{j=1}^{k_m} G(y, \bar x_{k_m, j}),  \sum_{j=1}^{k_m} G(y, \bar x_{k_m, j}) , \theta \bigr)\\
   =&Q\bigl(  G(y, \bar x_{k_m, 1}),  \sum_{j=2}^{k_m} G(y, \bar x_{k_m, j}) , \theta \bigr)+Q\bigl( \sum_{j=2}^{k_m} G(y, \bar x_{k_m, j}),   G(y, \bar x_{k_m, j1}) , \theta \bigr)+o_\theta(1)\\
   =&-2\int_{\partial B_\theta(\bar x_{k_m, 1})}\frac{\partial G(y, \bar x_{k_m, 1})}{\partial \nu}
   \frac{\partial }{\partial y_1}\sum_{j=2}^{k_m} G(y, \bar x_{k_m, j})+o_\theta(1)\\
   =&\frac2{(N-2)^2\omega_{N-1}^2} \int_{\partial B_\theta(\bar x_{k_m, 1})}\frac{N-2}{|y-\bar x_{k_m, 1}|^{N-1}}\sum_{j=2}^{k_m} \frac{(N-2)(\bar x_{k_m, j}-\bar x_{k, 1})_1}{|\bar x_{k_m, j}-\bar x_{k_m, 1}|^N}+o_\theta(1)\\
    =&\frac2{(N-2)^2\omega_{N-1}^2} (N-2)\omega_{N-1}\sum_{j=2}^{k_m} \frac{(N-2)(\bar x_{k_m, j}-\bar x_{k_m, 1})_1}{|\bar x_{k_m, j}-\bar x_{k_m, 1}|^N}+o_\theta(1).
   \end{split}
 \end{equation}

 We have $|\bar  x_{k_m, j}-\bar  x_{k_m, 1}|= 2\bar  r_{k_m} \sin\frac{j\pi}{k_m}$  and

 \[
 (\bar x_{k_m, j}-\bar  x_{k_m, 1})_1= \bar  r_{k_m}\cos\frac{2\pi j}{k_m}- \bar r_{k_m}= -2\bar  r_{k_m}\sin^2 \frac{\pi j}{k_m}=-\frac1{2\bar r_{k_m}} |\bar x_{k_m, j}-\bar x_{k_m, 1}|^2.
 \]
Therefore

\begin{equation}\label{12-31-5-21}
\begin{split}
  & Q(u_{k_m}, \xi_{k_m}, \delta)\\
 =& \frac{A_m B_m}{\bar \mu_{k_m}^{N-2}}\frac2{(N-2)^2\omega_{N-1}^2} (N-2)\omega_{N-1}\sum_{j=2}^{k_m} \frac{-(N-2)}{2\bar r_{k_m}|\bar x_{k_m, j}-\bar x_{k_m, 1}|^{N-2}}
  +O\bigl( \frac 1{(k_m\bar \mu_{k_m})^{2+\sigma}}\bigr).
 \end{split}
 \end{equation}

 Note that

 \[
 A:= \int_{\mathbb R^N} U^{2^*-1}= (N-2)\omega_{N-1} c_N,\quad B:=(2^*-1)\int_{\mathbb R^N} U^{2^*-2}\psi_0= -\frac{N-2}2 \int_{\mathbb R^N} U^{2^*-1}.
 \]
 Using  \eqref{2-31-5-21}, we obtain

\begin{equation}\label{12-31-5-21}
\begin{split}
   Q(u_{k_m}, \xi_{k_m}, \delta)
 =& \frac{b_0 AB}{\bar \mu_{k_m}^{2}k_m^2}\frac{2(N-2)\omega_{N-1} }
 {(N-2)^2\omega_{N-1}^2} \frac{ \frac{\Delta K(k_m^{-1}\bar x_{k_m, 1})}{2^*N} \int_{\mathbb R^N} U^{2^*} |y|^2}{r_{k_m}c_N \int_{\mathbb R^N} U^{2^*-1}}
  +o\bigl( \frac 1{k_m^2\bar \mu_{k_m}^{2}}\bigr)\\
  & =- \frac{b_0(N-2) \Delta K(k_m^{-1}\bar x_{k_m, 1}) \int_{\mathbb R^N} U^{2^*} |y|^2}{2^* Nr_{k_m} \bar \mu_{k_m}^{2}k_m^2}
  +o\bigl( \frac 1{k_m^2\bar \mu_{k_m}^{2}}\bigr).
 \end{split}
 \end{equation}
 So, we have proved

\begin{equation}\label{20-31-5-21}
   \begin{split}
 & -\int_{  B_\delta(\bar x_{k_m, 1})} u_{k_m}^{2^*-1} \xi_{k_m} \frac{\partial K_{k_m}(|y|) }{\partial y_1}\\
 =
 &- \frac{b_0(N-2) \Delta K(k_m^{-1}\bar x_{k_m, 1}) \int_{\mathbb R^N} U^{2^*} |y|^2}{2^* Nr_{k_m} \bar \mu_{k_m}^{2}k_m^2}
  +o\bigl( \frac 1{k_m^2\bar \mu_{k_m}^{2}}\bigr).
 \end{split}
 \end{equation}

 Now we estimate the left hand side of \eqref{20-31-5-21}.
 Since

 \[
 \nabla K_m( \bar x_{k_m, 1})= O(| k_m^{-1} |\bar x_{k_m, 1}|- r_0|)= O\bigl(\frac1{(k_m \bar \mu_{k_m})^{1+\sigma}}\bigr),
 \]
 we obtain

\begin{equation}\label{100-2-16-3}
   \begin{split}
 &\int_{  B_\delta(\bar x_{k_m, 1})} u_{k_m}^{2^*-1} \xi_{k_m} \frac{\partial K_{k_m}(|y|) }{\partial y_1}\\
 =& \int_{  B_\delta(\bar x_{k_m, 1})} u_{k_m}^{2^*-1} \xi_{k_m} \Bigl(\frac{\partial K_{k_m}(|y|) }{\partial y_1}- \frac{\partial K_{k_m}(\bar x_{k_m,1}) }{\partial y_1}
 \Bigr)+O\bigl(\frac1{(k_m \bar \mu_{k_m})^{1+\sigma}}\bigr)\\
 =& \int_{  B_\delta(\bar x_{k_m, 1})}  u^{2^*-1}_{k_m}\xi_m \Bigl( \Bigl\langle \nabla  \frac{\partial K_{k_m}(\bar x_{k_m, 1})}{\partial y_1}, y-\bar x_{k_m, 1}\Bigr\rangle
 +O\bigl(\frac1{(k_m \bar \mu_{k_m})^{1+\sigma}}\bigr)\\
 =&
 \int_{\mathbb R^N} U^{2^*-1} \bigl( b_{0}\psi_0 + b_{1} \psi_1\bigr)\Bigl( \Bigl\langle \nabla  \frac{\partial K(k_m^{-1}
 \bar x_{k_m, 1})}{\partial y_1}, \frac y{k_m\bar\mu_{k_m}}\Bigr\rangle+o\bigl(\frac1{k_m \bar \mu_{k_m}}\bigr)
 \\
 =& \frac{ K''(k_m^{-1}\bar x_{k_m, 1})   b_{ 1}}{k_m \bar \mu_{k_m}} \int_{\mathbb R^N} U^{2^*-1} \psi_1 y_1 + o\bigl(\frac1{k_m \bar \mu_{k_m}}\bigr).
 \end{split}
 \end{equation}

 Combining \eqref{20-31-5-21} and \eqref{100-2-16-3}, we obtain $b_{1}= 0$.

\textbf{ Step~2.}  Next, we use \eqref{1-26-2} for $\Omega=\mathbb R^N$ to obtain

\[
\int_{\mathbb R^N}  u_{k_m}^{2^*-1} \xi_{k_m}   \langle \nabla K_{k_m}(y), y\rangle
 = 0,
 \]
 which gives

\begin{equation}\label{3-17-3}
   \int_{\Omega_1}  u_{k_m}^{2^*-1} \xi_{k_m}   \langle \nabla K_{k_m}(y), y\rangle
 = 0,
 \end{equation}
where

\[
\Omega_j=\Bigl\{ y=(y',y'')\in\R^2\times \R^{N-2}:
 \Bigl\langle \frac {(y', 0)}{|y'|}, \frac{\bar x_{k_m, j}}{|\bar x_{k_m, j}|}\Bigr\rangle\ge \cos \frac{\pi}{k_m}\Bigr\}.
\]

On the other hand, we have

\begin{equation}\label{30-31-5-21}
\begin{split}
  & \int_{\Omega_1}  u_{k_m}^{2^*-1} \xi_{k_m}    \langle \nabla K_{k_m}(y), y\rangle= \int_{ B_\delta(\bar x_{k_m, 1})}  u_{k_m}^{2^*-1} \xi_{k_m}    \langle \nabla K_{k_m}(y), y\rangle
 + O\bigl( \frac 1{(k_m\bar \mu_{k_m})^{2+\sigma}}\bigr)\\
=& \int_{ B_\delta(\bar x_{k_m, 1})}  u_{k_m}^{2^*-1} \xi_{k_m}    \langle \nabla K_{k_m}(y), y- \bar x_{k_m, 1})
 \rangle+ r_{k_m} \int_{  B_\delta(\bar x_{k_m, 1})} u_{k_m}^{2^*-1} \xi_{k_m} \frac{\partial K_{k_m}(|y|) }{\partial y_1}\\
 &
 + O\bigl( \frac 1{(k_m\bar \mu_{k_m})^{2+\sigma}}\bigr)
 \end{split}
 \end{equation}
which gives

\begin{equation}\label{31-31-5-21}
\begin{split}
  &
 -r_{k_m} \int_{  B_\delta(\bar x_{k_m, 1})} u_{k_m}^{2^*-1} \xi_{k_m} \frac{\partial K_{k_m}(|y|) }{\partial y_1}\\
=& \int_{ B_\delta(\bar x_{k_m, 1})}  u_{k_m}^{2^*-1} \xi_{k_m}   \langle \nabla K_{k_m}(y), y- \bar x_{k_m, 1})
 \rangle
 + O\bigl( \frac 1{(k_m\bar \mu_{k_m})^{2+\sigma}}\bigr).
 \end{split}
 \end{equation}
But
\begin{equation}\label{4-17-3}
  \begin{split}
 &\int_{ B_\delta(\bar x_{k_m, 1})}  u_{k_m}^{2^*-1} \xi_{k_m}    \langle \nabla K_{k_m}(y), y- \bar x_{k_m, 1})
 \rangle\\
 = &\int_{ B_\delta(\bar x_{k_m, 1})}  u_{k_m}^{2^*-1} \xi_{k_m}   \langle \nabla K_{k_m}(y)-\nabla K_{k_m}(\bar x_{k_m, 1}
 ), y- \bar x_{k_m, 1})
 \rangle+ O\bigl( \frac 1{(k_m\bar \mu_{k_m})^{2+\sigma}}\bigr)\\
 = &\int_{ B_\delta(\bar x_{k_m, 1})}  u_{k_m}^{2^*-1} \xi_{k_m}   \langle \nabla^2 K_{k_m}(\bar x_{k_m, 1})(y-\bar x_{k_m, 1})
 ), y- \bar x_{k_m, 1})
 \rangle+ O\bigl( \frac 1{(k_m\bar \mu_{k_m})^{2+\sigma}}\bigr)\\
  = &\frac{ b_0 \Delta K(k_m^{-1}\bar x_{k_m, 1})}{N k_m^2\bar\mu_{k_m}^2}\int_{ \mathbb R^N}  U^{2^*-1}\psi_0  |y|^2+ o\bigl( \frac 1{
  k_m^2\bar \mu_{k_m}^{2}}\bigr)
 \end{split}
  \end{equation}

 Combining \eqref{20-31-5-21},  \eqref{31-31-5-21} and \eqref{4-17-3}, we obtain

  \begin{equation}\label{36-31-5-21}
  \begin{split}
 & - \frac{ b_0\Delta K(k_m^{-1}\bar x_{k_m, 1})}{N k_m^2\bar\mu_{k_m}^2}\int_{ \mathbb R^N}  U^{2^*-1}\psi_0  |y|^2\\
 =&
  \frac{b_0(N-2) \Delta K(k_m^{-1}\bar x_{k_m, 1}) \int_{\mathbb R^N} U^{2^*} |y|^2}{2^* N \bar \mu_{k_m}^{2}k_m^2}
  +o\bigl(\frac1{(k_m \bar \mu_{k_m})}\bigr).
  \end{split}
 \end{equation}

 Since

 \[
 \int_{ \mathbb R^N}  U^{2^*-1}\psi_0  |y|^2=-\frac2{2^*} \int_{\mathbb R^N} U^{2^*} |y|^2,
 \]
 and $  \frac2{2^*} \frac{ 1}{N }
 \ne
  \frac{N-2}{2^* N }$,  we see $b_0=0$.
\end{proof}

\begin{proof} [Proof of Theorem~\ref{th1-29-3}]
  We have
\begin{equation}\label{2-28-11}
\xi_{k_m} (y)= (2^*-1) \int_{\mathbb R^N}  \frac1{|z-y|^{N-2}} K_{k_m}(|z|)\bar u_{k_m}^{2^*-2}(z) \xi_{k_m}(z)\,dz.
\end{equation}

Now we estimate

\[
\begin{split}
&\Big|\int_{\mathbb R^N}  \frac1{|z-y|^{N-2}} K_{k_m}(|z|)\bar u_{k_m}^{2^*-2}(z) \xi_{k_m}(z)\,dz\Big|\\
\le  & C\|\xi _{k_m}\|_{*} \int_{\mathbb R^N} \frac1{|z-y|^{N-2}}\bar u_{k_m}^{2^*-2}(z) \sum_{j=1}^{k_m} \frac{\bar \mu_{k_m}^{\frac{N-2}2}}{ (1+ \bar \mu_{k_m} |z-\bar  x_{k_m, j}|)
^{\frac{N-2}2+\tau}}
\\
\le & C\|\xi_{k_m}\|_{*} \sum_{j=1}^{k_m} \frac{\bar \mu_{k_m}^{\frac{N-2}2}}{ (1+ \bar \mu_{k_m}|y-\bar  x_{k_m, j}|)
^{\frac{N-2}2+\tau+\theta}},
\end{split}
\]
for some $\theta>0$.
So we obtain
\[
\frac{ |\xi_{k_m}(y)|}{ \sum\limits_{j=1}^{k_m} \frac{\bar \mu_{k_m}^{\frac{N-2}2}}{ (1+\bar  \mu_{k_m}|y-\bar x_{k_m, j}|)
^{\frac{N-2}2+\tau}}
}\le C\|\xi_{k_m}\|_{*}\frac{ \sum\limits_{j=1}^{k_m} \frac{\bar \mu_{k_m}^{\frac{N-2}2}}{ (1+ \bar \mu_{k_m}|y-\bar x_{k_m, j}|)
^{\frac{N-2}2+\tau+\theta}}}{  \sum\limits_{j=1}^{k_m} \frac{\bar \mu_{k_m}^{\frac{N-2}2}}{ (1+ \bar \mu_{k_m}|y-\bar  x_{k_m, j}|)
^{\frac{N-2}2+\tau}}}.
\]

Since $\xi_{k_m}\to 0$ in $B_{R \bar \mu_{k_m}^{-1}}(\bar x_{k_m, j})$  and $\|\xi_{k_m}\|_*=1$, we know that
$$
\frac{ |\xi_{k_m}(y)|}{ \sum\limits_{j=1}^{k_m} \frac{\bar \mu_{k_m}^{\frac{N-2}2}}{ (1+ \bar \mu_{k_m}|y-\bar x_{k_m, j}|)
^{\frac{N-2}2+\tau}}
}
$$
attains its maximum in $\mathbb R^N \setminus \cup_{j=1}^{k_m} B_{R\bar  \mu_{k_m}^{-1}}(\bar x_{k_m, j})$. Thus
\[
\|\xi_{k_m}\|_* \le o(1)\|\xi_{k_m}\|_*.
\]
So $\|\xi_{k_m}\|_*\to 0$ as $m \to +\infty$. This is a contradiction to $\|\xi_{k_m}\|_*=1$.

\end{proof}

\section{Proof of the main result}
\setcounter{equation}{0}

Let  $u_k$ be the $k$-bubbling solutions in Theorem~A, where $k>0$ is a large even integer. Since $k$ is even, $u_k$ is even in each
$y_j$, $j=1,\cdots, N$.   Moreover,  $u_k$  is radial in $y''=(y_3, \cdots, y_N)$.

Let $n\ge k$ be a large even integer.
Set
\[
p_j=\Bigl(0, 0, t \cos\frac{2(j-1)\pi}n, t\sin\frac{2(j-1)\pi}n,0\Bigr),\quad j=1,\cdots,n,
\]
where   $t $ is close to $r_0$.

  Define
\[
\begin{split}
X_s=\Bigl\{ u: & u\in H_s, u\;\text{is even in} \;y_h, h=1,\cdots,N,\\
& u(y_1, y_2, t\cos\theta , t\sin\theta, y^*)=
u(y_1,y_2, t\cos(\theta+\frac{2\pi j}n) , t\sin(\theta+\frac{2\pi j}n), y^*)
\Bigr\}.
\end{split}
\]
Here $y^*=(y_5,\cdots,y_N)$.

Let
\[
D_j=\Bigl\{ y=(y',y_3, y_4, y^*)\in\R^2\times \mathbb R^2\times \R^{N-4}:
 \Bigl\langle \frac {(0, 0, y_3, y_4, 0,\cdots,0)}{|(y_3, y_4)|}, \frac{p_{ j}}{|p_{ j}|}\Bigr\rangle\ge \cos \frac{\pi}{n}\Bigr\}.
\]

Note that   both $u_k $  and $\sum_{j=1}^n U_{p_j, \lambda}$ belong to $ X_s$, while  $u_k$  and $\sum_{j=1}^n U_{p_j, \lambda}$ are separated  from each other. We aim to construct a solution for \eqref{1.4} of the form
\[
u= u_k +\sum_{j=1}^n U_{p_j,\lambda} +\xi,
\]
where $\xi\in X_s$ is a small perturbed term.

We define the linear operator
\begin{equation}\label{10-29-3}
 Q_n \xi= -\Delta \xi -(2^*-1)K(|y|)\Bigl(  u_k+\sum_{j=1}^n U_{p_j,\lambda}\Bigr)^{2^*-2}\xi,\quad \xi\in X_s.
 \end{equation}

Denote
 \[
Z_{j, 1}=\frac{\partial U_{p_j, \lambda }}{\partial r},\;\;j=1,\cdots,k,\quad Z_{j, 2}=\frac{\partial U_{p_j, \lambda }}{\partial \lambda}
  \]

 Let  $h_n\in X_s$. Consider
\begin{equation}\label{20-2-4}
\begin{cases}
Q_n \xi_n = h_n + \sum\limits_{i=1}^2 a_{n, i} \sum\limits_{j=1}^n  U_{p_j,\lambda}^{2^*-2}  Z_{j, i}, \\
\xi_n \in X_s,\\
\displaystyle\int_{\mathbb R^N} U_{p_j,\lambda}^{2^*-2} Z_{j, i} \xi_n=0,\;\;i=1, 2,\; j=1,\cdots, n,
\end{cases}
\end{equation}
for some constants $a_{n,i}$, depending on $\xi_n$.

We define

\[
\| u\|_{*,n}= \sup_{y\in \mathbb R^N} |u(y)|\Bigl(\sum_{j=1}^n \frac{\lambda^{\frac{N-2}2}}{ ( 1+ \lambda |y-p_j|)^{\frac{N-2}2+\tau}}\Bigr)^{-1},
\]
and
\[
\| f\|_{**,n}= \sup_{y\in \mathbb R^N} |f(y)|\Bigl(\sum_{j=1}^n \frac{\lambda^{\frac{N+2}2}}{ ( 1+ \lambda |y-p_j|)^{\frac{N+2}2+\tau}}\Bigr)^{-1},
\]
where $\tau=\frac{N-4}{N-2}$.

\begin{lemma}\label{l20-2-4}
Suppose that $N\ge 7$.
Assume that $\xi_n$ solve \eqref{20-2-4}.  If $\|h_n\|_{**,n }\to 0$, then $\|\xi_n\|_{*,n}\to 0$.
\end{lemma}

\begin{proof}

Suppose that there are $\xi_n$ and $h_n$, satisfying \eqref{20-2-4}, $\|h_n\|_{**,n }\to 0$ and $\|\xi_n\|_{*,n}=1$.
We write

\[
L_k \xi_n =  (2^*-1)K(|y|)\Bigl[\Bigl(  u_k+\sum_{j=1}^n U_{p_j,\lambda}\Bigr)^{2^*-2}-
u_k^{2^*-2}\Bigr]+ h_n + \sum\limits_{i=1}^2 a_{n, i} \sum\limits_{j=1}^n    Z_{j, i}
\]
Then,
\[
\begin{split}
\xi_n(x)
 =\int_{\mathbb R^N} G_k(y, x)\Bigl((2^*-1) K(|y|)\Bigl[&\Bigl(  u_k+\sum_{j=1}^n U_{p_j,\lambda}\Bigr)^{2^*-2}-
u_k^{2^*-2}\Bigr] \xi_n\\
&+h_n + \sum\limits_{i=1}^2 a_{n, i} \sum\limits_{j=1}^n  U_{p_j,\lambda}^{2^*-2}  Z_{j, i}\Bigr)\,dy.
\end{split}
\]

 For $|x|\le R$, by Proposition~\ref{p1-25-10n},

\[
\begin{split}
|\xi_n(x)| \le &C\int_{\mathbb R^N} G_k(y, x)\Bigl((\sum_{j=1}^n U_{p_j,\lambda}\Bigr)^{2^*-2}| \xi_n|+|h_n |+ |\sum\limits_{i=1}^2 a_{n, i} \sum\limits_{j=1}^n   U_{p_j,\lambda}^{2^*-2}  Z_{j, i}|\Bigr)\,dy\\
\le &C\int_{\mathbb R^N} \frac1{|y-x|^{N-2}}\Bigl((\sum_{j=1}^n U_{p_j,\lambda}\Bigr)^{2^*-2}| \xi_n|+|h_n |+ |\sum\limits_{i=1}^2 a_{n, i} \sum\limits_{j=1}^n   U_{p_j,\lambda}^{2^*-2}  Z_{j, i}|\Bigr)\,dy.
\end{split}
\]
Then, similar to the proof of Lemma~2.1 in \cite{WY}, we can prove

\begin{equation}\label{20-1-6-21}
\Bigl(\sum_{j=1}^n \frac{\lambda^{\frac{N-2}2}}{ ( 1+ \lambda |y-p_j|)^{\frac{N-2}2+\tau}}\Bigr)^{-1}|\xi_n(x)| \le C\|h_n\|_{**,n}+
\frac{ \sum_{j=1}^n \frac{C}{ ( 1+ \lambda |y-p_j|)^{\frac{N-2}2+\tau+\theta}} }{  \sum_{j=1}^n \frac{1}{ ( 1+ \lambda |y-p_j|)^{\frac{N-2}2+\tau}}}
\|\xi_n\|_{*,n},
\end{equation}
for some $\theta>0$.

Now, we discuss the case $|x|\ge R$.   We have

\[
\begin{split}
|\xi_n(x)| \le &C\int_{\mathbb R^N} \frac1{|y-x|^{N-2}}\Bigl((\sum_{j=1}^n U_{p_j,\lambda}\Bigr)^{2^*-2}| \xi_n|+|h_n |+ |\sum\limits_{i=1}^2 a_{n, i} \sum\limits_{j=1}^n    Z_{j, i}|\Bigr)\,dy\\
&+C\int_{\mathbb R^N} \frac1{|y-x|^{N-2}}u_k^{2^*-2}| \xi_n|.
\end{split}
\]
We have

\[
\int_{\mathbb R^N} \frac1{|y-x|^{N-2}}u_k^{2^*-2}| \xi_n|\le
\|\xi_n\|_{*, n} \int_{\mathbb R^N} \frac1{|y-x|^{N-2}}u_k^{2^*-2} \sum_{j=1}^n \frac{\lambda^{\frac{N-2}2}}{ ( 1+ \lambda |y-p_j|)^{\frac{N-2}2+\tau}}.
\]
Let  $d_j= \frac12 |x-p_j|$. Then, we have

\[
\begin{split}
&\int_{ B_{d_j}(p_j)} \frac1{|y-x|^{N-2}}u_k^{2^*-2}  \frac{\lambda^{\frac{N-2}2}}{ ( 1+ \lambda |y-p_j|)^{\frac{N-2}2+\tau}}\\
\le &
\frac{C}{   \lambda^\tau d_j^{N-2}}\int_{ B_{d_j}(p_j)}
 \frac{1}{ |y-p_j|^{\frac{N-2}2+\tau}}\frac1{(1+|y|)^4}\\
 \le &\frac{C}{   \lambda^\tau d_j^{\frac{N-2}2+\tau+2}}\le \frac{C\lambda^{\frac{N-2}2}}{ ( 1+ \lambda |x-p_j|)^{\frac{N-2}2+\tau}}\frac1{|x|^2},
\end{split}
\]
if $ N>6+2\tau $,

\[
\int_{ B_{d_j}(p_j)} \frac1{|y-x|^{N-2}}u_k^{2^*-2}  \frac{\lambda^{\frac{N-2}2}}{ ( 1+ \lambda |y-p_j|)^{\frac{N-2}2+\tau}}\\
\le
\frac{C\lambda^{\frac{N-2}2}}{ ( 1+ \lambda |x-p_j|)^{\frac{N-2}2+\tau}}\frac{\ln |x|}{|x|^2},
\]
if $ N=6+2\tau $,
  and
\[
\begin{split}
&\int_{ B_{d_j}(p_j)} \frac1{|y-x|^{N-2}}u_k^{2^*-2}  \frac{\lambda^{\frac{N-2}2}}{ ( 1+ \lambda |y-p_j|)^{\frac{N-2}2+\tau}}\\
\le&
\frac{C}{   \lambda^\tau d_j^{N-2}}
 \le  \frac{C\lambda^{\frac{N-2}2}}{ ( 1+ \lambda |x-p_j|)^{\frac{N-2}2+\tau}}\frac1{|x|^{\frac{N-2}2-\tau}},
\end{split}
\]
if $ N<6+2\tau $.

We also have
\[
\begin{split}
&\int_{\mathbb R^N\setminus B_{d_j}(x)} \frac1{|y-x|^{N-2}}u_k^{2^*-2}  \frac{\lambda^{\frac{N-2}2}}{ ( 1+ \lambda |y-p_j|)^{\frac{N-2}2+\tau}}\\
\le &
\frac{C\lambda^{\frac{N-2}2}}{ ( 1+ \lambda |x-p_j|)^{\frac{N-2}2+\tau}}\int_{\mathbb R^N\setminus  B_{d_j}(x)} \frac1{|y-x|^{N-2}}\frac1{(1+|y|)^4}
\\
\le &\frac{C\lambda^{\frac{N-2}2}}{ ( 1+ \lambda |x-p_j|)^{\frac{N-2}2+\tau}}\frac1{|x|^2}.
\end{split}
\]
So, we have that for $|x|\ge R$,

\begin{equation}\label{21-1-6-21}
\begin{split}
&\Bigl(\sum_{j=1}^n \frac{\lambda^{\frac{N-2}2}}{ ( 1+ \lambda |y-p_j|)^{\frac{N-2}2+\tau}}\Bigr)^{-1}|\xi_n(x)|\\
 \le &C\|h_n\|_{**,n}+
\frac{C}{|x|^\sigma}\|\xi_n\|_{*,n}
+
\frac{ \sum_{j=1}^n \frac{C}{ ( 1+ \lambda |y-p_j|)^{\frac{N-2}2+\tau+\theta}} }{  \sum_{j=1}^n \frac{1}{ ( 1+ \lambda |y-p_j|)^{\frac{N-2}2+\tau}}}
\|\xi_n\|_{*,n},
\end{split}
\end{equation}
for some $\sigma>0$ and $\theta>0$.

Combining \eqref{20-1-6-21} and \eqref{21-1-6-21}, we see that the maximum of the function in the left hand side of \eqref{20-1-6-21}
can only be achieved $B_{\lambda^{-1}L}(p_1)$ for some large constant $L>0$. Using the last relation in \eqref{20-2-4},
we will deduce a contradiction.

\end{proof}

 From now on, we assume that $N\ge 7$.
We want to construct a solution $u$ for \eqref{1.4}  with
\[
u= u_k +\sum_{j=1}^n  U_{p_j,\lambda} +\omega,
\]
where $\omega\in X_s$ is a small perturbed term, satisfying
\[
 \int_{\mathbb R^N}  U_{p_j,\lambda}^{2^*-2} Z_{j,l} \omega=0,\quad j=1,\cdots, n,\,l=1,2.
 \]
 Then $\omega$ satisfies
\begin{equation}\label{3-22-12}
  Q_n \xi_n = l_n+R(\xi_n),
 \end{equation}
where
\begin{equation}\label{4-22-12}
   l_n =K(|y|) \Bigl(  u_k+\sum_{j=1}^n U_{p_j,\lambda}\Bigr)^{2^*-1} - K(|y|) u_k^{2^*-1} -\sum_{j=1}^n U_{p_j,\lambda}^{2^*-1},
 \end{equation}
and
\begin{equation}\label{5-22-12}
 \begin{split}
   R_n(\xi) = & K(|y|) \Bigl(  u_k+\sum_{j=1}^n U_{p_j,\lambda}+\xi\Bigr)^{2^*-1}-K(|y|)\Bigl(  u_k+\sum_{j=1}^n U_{p_j,\lambda}\Bigr)^{2^*-1}\\
   &-(2^*-1) K(|y|)\Bigl(  u_k+\sum_{j=1}^n U_{p_j,\lambda}\Bigr)^{2^*-2}\xi.
   \end{split}
 \end{equation}

\begin{lemma}\label{l1-22-12}
\[
\|R_n(\xi)\|_{**,n}\le C\|\xi\|_{**,n}^{2^*-1}.
\]
Moreover,
\[
\|R_n(\xi_1)- R_n(\xi_2)\|_{**,n}\le C \Bigl(  \|\xi_1\|_{**,n}^{2^*-2} +\|\xi_2\|_{**,n}^{2^*-2}\Bigr)\|\xi_1-\xi_2\|_{**,n}.
\]

\end{lemma}

\begin{proof}
 For $N\ge 7$, we have

 \[
| R_n(\xi)|\le C|\xi|^{2^*-1},
\]
and

\[
|R_n(\xi_1)- R_n(\xi_2)|\le C( |\xi_1|^{2^*-2}+|\xi_2|^{2^*-2})|\xi_1-\xi_2|.
\]
So we can prove  this lemma as in Lemma~2.4 in \cite{WY}.
\end{proof}

We have the following estimate for $\|l_n\|_{**,n}$.
\begin{lemma}\label{c4-l1-25-3}
There is a small $\sigma>0$, such that
\[
\|l_n\|_{**, n}\le \frac{C }{\lambda^{1+\sigma}}.
\]

\end{lemma}

\begin{proof}

Write
\[
\begin{split}
l_n =& K(|y|)\Bigl(\bigl(  u_k+\sum_{j=1}^n U_{p_j,\lambda}\bigr)^{2^*-1} - u_k^{2^*-1} -\bigl(\sum_{j=1}^n U_{p_j,\lambda}\bigr)^{2^*-1}\Bigr) \\
&+\bigl(
\sum_{j=1}^n U_{p_j,\lambda}\bigr)^{2^*-1}-\sum_{j=1}^n U_{p_j,\lambda}^{2^*-1}\\
& +  \bigl( K(|y|)-1\bigr)\bigl(  \sum_{j=1}^n U_{p_j,\lambda}\bigr)^{2^*-1}\\
:= & J_1+ J_2+J_3.
\end{split}
\]

 Similar to the proof of Lemma~2.5  in \cite{WY}, we can prove

 \[
 \|J_2\|_{**,n},\;  \|J_3\|_{**,n}\le \frac{C }{\lambda^{1+\sigma}}.
\]

Without loss of generality, we assume that $y\in D_1$.
First, we study the case  $y\in D_1\cap B_{ \lambda^{-\frac12}}(p_1)$. We have
\[
 u_k\le C_1\le  C U_{p_1,\lambda}.
 \]
Thus,

\[
 U^{2^*-2}_{p_1, \lambda}\ge C'.
\]
Noting that $2^*-2<1$, we find that
\[
\begin{split}
|J_1|\le & C  \bigl( \sum_{j=1}^n U_{p_j,\lambda}\bigr)^{2^*-2}u_k  + u_k^{2^*-1}\le C\bigl( \sum_{j=1}^n U_{p_j,\lambda}\bigr)^{2^*-2}+C\\
\le& C   \sum_{j=1}^n U_{p_j,\lambda}^{2^*-2}\le \sum_{j=1}^n\frac{C\lambda^2}{(1+\lambda |y-p_j|)^4}
\\
\le &\frac C{\lambda^{1+\sigma}} \sum_{j=1}^n \frac{\lambda^{\frac{N+2}2}}{ ( 1+ \lambda |y-p_j|)^{\frac{N+2}2+\tau}},
\end{split}
\]
since

\[
\lambda^{\frac{N-2}2-1-\sigma} \ge \lambda^{\frac{N-2}2-2+\tau}\ge c_0(1+\lambda|y-p_j|)^{\frac{N-2}2-2+\tau}.
\]

 Now, we consider  $y\in D_1\setminus  B_{ \lambda^{-\frac12}}(p_1)$. We use

 \[
 |(1+t)^{2^*-1}-1-(2^*-1)t|\le C|t|^{2^*-1},
 \]
 to deduce

 \[
|J_1|\le  C u_k^{2^*-2}\sum_{j=1}^n U_{p_j,\lambda}   + C\bigl( \sum_{j=1}^n U_{p_j,\lambda}\bigr)^{2^*-1}.
\]
We also have

 \[
\begin{split}
& \bigl( \sum_{j=1}^n U_{p_j,\lambda}\bigr)^{2^*-1}\le U_{p_1, 1}^{2^*-1} + CU_{p_1, 1}^{2^*-2} \sum_{j=2}^n U_{p_j,\lambda}
+C\bigl( \sum_{j=2}^n U_{p_j,\lambda}\bigr)^{2^*-1}\\
\le &
C  \sum_{j=1}^n U_{p_j,\lambda}
+C\bigl( \sum_{j=2}^n U_{p_j,\lambda}\bigr)^{2^*-1}
\end{split}
\]
since $ U_{p_1, 1}\le C$ if  $y\in D_1\setminus  B_{ \lambda^{-\frac12}}(p_1)$.

It is easy to check that

\[
 \sum_{j=1}^n U_{p_j,\lambda}\le \frac C{\lambda^{1+\sigma}} \sum_{j=1}^n \frac{\lambda^{\frac{N+2}2}}{ ( 1+ \lambda |y-p_j|)^{\frac{N+2}2+\tau}}.
 \]

On the other hand,  from

\[
 \frac{1}{ ( 1+ \lambda |y-p_j|)^{N-2}}\le \frac{C}{ ( 1+ \lambda |y-p_1|)^{\frac{N-2}2+\frac{\tau}{2^*-1}}}
 \frac{1}{ (  \lambda |p_j-p_1|)^{\frac{N-2}2-\frac{\tau}{2^*-1}}}, \quad y\in D_1\setminus  B_{ \lambda^{-\frac12}}(p_1),
 \]
we find

 \[
\begin{split}
&\bigl( \sum_{j=2}^n U_{p_j,\lambda}\bigr)^{2^*-1}\le  \frac{C\lambda^{\frac{N+2}2}}{ ( 1+ \lambda |y-p_1|)^{\frac{N+2}2+\tau}} \bigl( \sum_{j=1}^n \frac{1}{ ( \lambda |p_1-p_j|)^{\frac{N-2}2-\frac\tau{2^*-1}}}\bigr)^{2^*-1}\\
\le & C\bigl(\frac n\lambda\bigr)^{\frac{N+2}2-\tau}\frac{\lambda^{\frac{N+2}2}}{ ( 1+ \lambda |y-p_1|)^{\frac{N+2}2+\tau}}
\le  C\bigl(\frac 1\lambda\bigr)^{\frac{N+2}{N-2}-\frac{2\tau}{N-2}}\frac{\lambda^{\frac{N+2}2}}{ ( 1+ \lambda |y-p_1|)^{\frac{N+2}2+\tau}}.
\end{split}
\]
It is easy to check that $\frac{N+2}{N-2}-\frac{2\tau}{N-2}= 1+\frac{4-2\tau}{N-2}>1$.  Thus, we have proved
\[
|J_1|\le
C  \sum_{j=1}^n U_{p_j,\lambda}\le \frac C{\lambda^{1+\sigma}} \sum_{j=1}^n \frac{\lambda^{\frac{N+2}2}}{ ( 1+ \lambda |y-p_j|)^{\frac{N+2}2+\tau}},
\]
which implies

\[
\|J_1\|_{**,n}\le \frac{C }{\lambda^{1+\sigma}}.
\]

\end{proof}

 We consider the following problem:
\begin{equation}\label{80-3-4}
\begin{cases}
Q_n \xi_n =  l_n + R_n (\xi_n)+ \sum\limits_{i=1}^2 a_{n, i} \sum\limits_{j=1}^n    U_{p_j,\lambda}^{2^*-2}  Z_{j, i}, \\
\xi_n \in X_s,\\
\displaystyle\int_{\mathbb R^N} U_{p_j,\lambda}^{2^*-2} Z_{j,l} \xi_n=0,\;\; j=1,\cdots, n,\,l=1,2.
\end{cases}
\end{equation}

Using Lemmas~\ref{l20-2-4}, \ref{c4-l1-25-3} and \ref{l1-22-12}, we can prove the following proposition in a standard way.
\begin{proposition}\label{p1-6-3}
There is an integer $n_0>0$, such that for each $n\ge n_0$  and  $(t,\lambda) \in (r_0-\delta,\, r_0+\delta)\times [\lambda_0 n^{\frac{N-2}{N-4}}, \lambda_1
 n^{\frac{N-2}{N-4}}
 ]$ ,  \eqref{80-3-4} has a solution $\xi_n$ for some constants
$a_{n, i}$.   Moreover,  $\xi_n$
is a $C^1$ map from $(r_0-\delta,\, r_0+\delta)\times  [\lambda_0 n^{\frac{N-2}{N-4}}, \lambda_1
 n^{\frac{N-2}{N-4}}
 ]$ to $X_s$,
and
\[
\|\xi_n\|_{**,n}\le  \frac{C }{\lambda^{1+\sigma}}
\]
for some $\sigma>0$.

\end{proposition}

Define
\[
I(u)=\frac12\int_{\R^N}  |Du|^2-\frac1{2^*}\int_{\R^N} K(|y|) |u|^{2^*}.
\]
Let
\[
F(t,\lambda)= I\Bigl( u_{k}+\sum_{j=1}^n U_{p_j,\lambda}
+\xi_{n}\Bigr).
\]
To obtain a solution of the form $ u_{k}+\sum_{j=1}^n U_{p_j,\lambda}
+\xi_{n}$, we just need to find a critical point for $F(t,\lambda)$ in $[r_0-\delta, r_0+\delta]\times [ \lambda_0 n^{\frac{N-2}{N-4}}, \lambda_1 n^{\frac{N-2}{N-4}}
 ]$.

\begin{proof}[Proof of Theorem~\ref{th11}]

 We have
\begin{equation}\label{60-22-12}
 \begin{split}
 F(t, \lambda)=&  I\Bigl( u_{k}+\sum_{j=1}^n U_{p_j,\lambda}
\Bigr)+ n  O\bigl(  \frac{1 }{\lambda^{2+\sigma}}\bigr).
\end{split}
 \end{equation}

 On the other hand,
\begin{equation}\label{1-20-4}
 \begin{split}
   I\Bigl( u_{k}+\sum_{j=1}^n U_{p_j,\lambda}
\Bigr)=&   I\Bigl( \sum_{j=1}^n U_{p_j,\lambda}\Bigr) +I( u_k)  +\frac12 \sum_{j=1}^n\int_{\mathbb R^N} K(|y|)u_k^{2^*-1} U_{p_j,\lambda}\\
&- \frac1{2^*} \int_{\mathbb R^N} K(|y|) \Bigl( \big(u_{k}+\sum_{j=1}^n U_{p_j,\lambda}\big)^{2^*}-\big(\sum_{j=1}^n U_{p_j,\lambda}\big)^{2^*}- u_k^{2^*}
\Bigr).
\end{split}
 \end{equation}

 It is easy to check
\[
 \int_{\mathbb R^N}K(|y|) u_k^{2^*-1} U_{p_j,\lambda}= O\Bigl(\frac1{\lambda^{\frac{N-2}2}}\Bigr).
 \]

  For $y\in \cup_{j=1}^n\bigl(D_j  \setminus   B_{ \lambda^{-\frac12}}(p_j)\bigr)$, we have
 
   \[
  \begin{split}
 &\Bigl| \big(u_{k}+\sum_{j=1}^n U_{p_j,\lambda}\big)^{2^*}- u_k^{2^*}-\big(\sum_{j=1}^n U_{p_j,\lambda}\big)^{2^*}\Bigr|\\
 \le & C u_k^{2^*-1} \sum_{j=1}^n U_{p_j,\lambda}+C\big(\sum_{j=1}^n U_{p_j,\lambda}\big)^{2^*}.
 \end{split}
  \]
 Thus,
  
  \[
  \begin{split}
 & \int_{\cup_{j=1}^n(D_j  \setminus   B_{ \lambda^{-\frac12}}(p_j))}\Bigl| \big(u_{k}+\sum_{j=1}^n U_{p_j,\lambda}\big)^{2^*}- u_k^{2^*}-\big(\sum_{j=1}^n U_{p_j,\lambda}\big)^{2^*}\Bigr|\\
 \le & C \int_{\mathbb R^N}u_k^{2^*-1} \sum_{j=1}^n U_{p_j,\lambda}+C\int_{\cup_{j=1}^n(D_j  \setminus   B_{ \lambda^{-\frac12}}(p_j))}\big(\sum_{j=1}^n U_{p_j,\lambda}\big)^{2^*}\\
 \le & \frac{Cn}{\lambda^{\frac{N-2}2}}+Cn\int_{D_1  \setminus   B_{ \lambda^{-\frac12}}(p_1)}\big(\sum_{j=1}^n U_{p_j,\lambda}\big)^{2^*}.
 \end{split}
  \] 
  
  For $y\in D_1  \setminus   B_{ \lambda^{-\frac12}}(p_1)$, 
  \[
  \begin{split}
 \big(\sum_{j=1}^n U_{p_j,\lambda}\big)^{2^*}
 \le & C\Bigl( \frac1{\lambda^{\frac{N-2}2}|y-p_1|^{N-2}}+\frac1{\lambda^{\frac{N-2}2-\tau}|y-p_1|^{N-2-\tau}}\Bigr)^{2^*}\\
 \le & C\Bigl( \frac1{\lambda^{N}|y-p_1|^{2N}
 }+\frac1{\lambda^{N-2^*\tau}|y-p_1|^{2N-2^* \tau}}\Bigr),
 \end{split}
  \]
  which gives
  
   \[
   \int_{D_1  \setminus   B_{ \lambda^{-\frac12}}(p_1)}\big(\sum_{j=1}^n U_{p_j,\lambda}\big)^{2^*}\\
\le  \frac{C
}{\lambda^{\frac{N}2}}+\frac{C
}{\lambda^{\frac{1}2(N-2^*\tau)}}
 \le \frac{C
}{\lambda^{\frac{N-2}2}}.
  \] 
  So we have proved 
\[
  \int_{\cup_{j=1}^n(D_j  \setminus   B_{ \lambda^{-\frac12}}(p_j))}
  K(|y|)\Bigl| \big(u_{k}+\sum_{j=1}^n U_{p_j,\lambda}\big)^{2^*}- u_k^{2^*}-\big(\sum_{j=1}^n U_{p_j,\lambda}\big)^{2^*}\Bigr|\\
 =  O\Bigl(\frac n{\lambda^{\frac{N-2}2
 }}\Bigr).
  \]

On the other hand, we have

\[
 \begin{split}
 & \int_{\cup_{j=1}^n(D_j  \cap   B_{ \lambda^{-\frac12}}(p_j))}
  K(|y|)\Bigl| \big(u_{k}+\sum_{j=1}^n U_{p_j,\lambda}\big)^{2^*}- u_k^{2^*}-\big(\sum_{j=1}^n U_{p_j,\lambda}\big)^{2^*}\Bigr|\\
 =  & n\int_{D_1  \cap   B_{ \lambda^{-\frac12}}(p_1)}
  K(|y|)\Bigl| \big(u_{k}+\sum_{j=1}^n U_{p_j,\lambda}\big)^{2^*}- u_k^{2^*}-\big(\sum_{j=1}^n U_{p_j,\lambda}\big)^{2^*}\Bigr|.
 \end{split}
  \]
  
  It holds
\[
  \int_{  D_1\cap   B_{ \lambda^{-\frac12}}(p_1)}
  K(|y|) u_k^{2^*}=O\Bigl(\frac1{\lambda^{\frac{N}2}}\Bigr),
  \]
and
\[
  \begin{split}
 & \int_{  D_1\cap   B_{ \lambda^{-\frac12}}(p_1)}
  K(|y|)\Bigl| \big(u_{k}+\sum_{j=1}^n U_{p_j,\lambda}\big)^{2^*}-\big(\sum_{j=1}^n U_{p_j,\lambda}\big)^{2^*}\Bigr|\\
  \le & C \int_{  D_1\cap   B_{ \lambda^{-\frac12}}(p_1)}\Bigl(
  \bigl(\sum_{j=1}^n U_{p_j,\lambda}\bigr)^{2^*-1}u_k+u_k^{2^*}\Bigr)\\
  \le &  C \int_{  D_1\cap   B_{ \lambda^{-\frac12}}(p_1)}
  \Bigl( U_{p_1,\lambda}^{2^*-1} +  \frac{\lambda^{\frac{N+2}2}}{(1+ \lambda|y-p_1|)^{(2^*-1)(N-2)(1-\tau_1)   }}\Bigr)+ \frac{C}{\lambda^{\frac{N}2}}\\
  \le & \frac{C}{\lambda^{\frac{N-2}2}},
 \end{split}
  \]
where $\tau_1=\frac{N-4}{(N-2)^2}$.

In conclusion, we obtain

\begin{equation}\label{2-20-4}
   I\Bigl( u_{k}+\sum_{j=1}^n U_{p_j,\lambda}
\Bigr)=   I\Bigl( \sum_{j=1}^n U_{p_j,\lambda}\Bigr) +I( u_k)  +O\bigl(\frac n{\lambda^{\frac{N-2}2}}\bigr).
 \end{equation}

Combining \eqref{60-22-12}  and \eqref{2-20-4}, and  proceeding  as in \cite{WY}, we obtain
\begin{equation}\label{3-20-4}
 \begin{split}
 F(t, \lambda)=&   I\Bigl( \sum_{j=1}^n U_{p_j,\lambda}\Bigr) +I( u_k)+ n  O\bigl(  \frac{1 }{\lambda^{2+\sigma}}\bigr)\\
 =& I(u_k)+nA+n\Bigl( \frac{B_1}{\lambda^2} + \frac{B_2}{\lambda^2} (\lambda r_0 -t)^2-  \frac{B_3 n^{N-2}}{\lambda^{N-2}}\Bigr)\\
 &+ n  O\Bigl( \frac{B_1}{\lambda^{2+\sigma}} + \frac{B_2}{\lambda^2} (\lambda r_0 -t)^3\Bigr),
\end{split}
 \end{equation}
  where $A = \frac12\int_{\mathbb R^N}  |\nabla U_{0, 1}|^2 -\frac1{2^*}\int_{\mathbb R^N} U_{0, 1}^{2^*}$, $B_1$, $B_2$  and $B_3$ are some positive constants,
 and $\sigma>0$ is a small constant.\\

Now  to find a critical point for $F(t,
\lambda)$, we just need to  proceed exactly as in \cite{WY}.
\end{proof}

\appendix

\section{the Green's functions}

 For any function $f$ defined in $\mathbb R^N$, we define its corresponding function $f^*\in H_s\cap D^{1,2}(\mathbb R^N)$ as follows.

 We first define $A_j$ as

 \[
 A_j z= \bigl( r \cos(\theta+\frac{2j \pi}k),  r \sin(\theta+\frac{2j \pi}k), z''),\quad j=1, \cdots, k,
 \]
 where $z= (z', z'')\in \mathbb R^N$, $z'= (r\cos\theta, r\sin \theta)\in \mathbb R^2$, $z''\in \mathbb R^{N-2}$, while

 \[
 B_i z= \bigl( z_1,\cdots, z_{i-1}, -z_i, z_{i+1},\cdots, z_N),\quad i=1, \cdots, N.
 \]
 Let

 \[
 \bar f(y)=\frac1 k\sum_{j=1}^k f(A_j y),
 \]
 and

 \[
 f^*(y) = \frac1{(N-1)}\sum_{i=2}^N \frac12 \bigl( \bar f(y)+ \bar f(B_i y)\bigr).
 \]
 Then $f^*\in H_s\cap D^{1,2}(\mathbb R^N)$.

Let

\begin{equation}\label{1-25-10n}
 L_k \xi= -\Delta \xi   -(2^*-1)K(|y|) u_k^{2^*-2}\xi.
 \end{equation}
 In this section, we discuss the Green's function of $L_k$. Since $\delta_x$ is not in $H_s\cap D^{1,2}(\mathbb R^N)$, we consider

\begin{equation}\label{2-25-10n}
 L_k u =  \delta_x^*,  \quad u\in H_s.
 \end{equation}
 The solution of \eqref{2-25-10n} is denoted as $G_k(y, x)$. Let us point out that

\[
  \delta_x^* =\frac1 {N-1}\sum_{i=2}^N \frac12 \Bigl( \frac1 k\sum_{j=1}^k \delta_{A_j x}+ \frac1 k\sum_{j=1}^k \delta_{B_i A_j x}\Bigr).
 \]

 \begin{proposition}\label{p1-25-10n}

 The solution $G_k(y, x)$ of \eqref{2-25-10n} satisfies

 \[
 |G_k(y, x)|\le  \frac1 {N-1}\sum_{i=2}^N \frac12 \Bigl( \frac1 k\sum_{j=1}^k \frac{C}{|y- A_j x|^{N-2}}+ \frac1 k\sum_{j=1}^k
 \frac{C}{|y- B_i A_j x|^{N-2}}\Bigr)
 \]
 for all $x\in B_R(0)$, where $R>0$ is any fixed large constant.

 \end{proposition}

 \begin{proof}

 Let $v_1= \frac{C_N}{ |y-x|^{N-2}}$, which satisfies $-\Delta v_1= \delta_x$ in $\mathbb R^N$.  Let $v_2$ be the solution of

 \[
 \begin{cases}
 -\Delta v= (2^*-1)K(y) u_k^{2^*-2} v_1,& \text{in}\; B_{2R}(0),\\
 v=0,& \text{on}\; \partial B_{2R}(0).
 \end{cases}
 \]
 Then $v_2\ge 0$ and

 \[
  v_2(y)= \int_{B_{2R}(0)}G(z, y) (2^*-1) u_k^{2^*-2} K(y)v_1\le C \int_{B_R(0)}\frac1{|y-z|^{N-2}}\frac1{|z-x|^{N-2}}\,dz\le \frac{C}{|y-x|^{N-4}},
 \]
 where $G(z,y)$ is the Green's function of $-\Delta$ in $B_{2R}(0)$ with zero boundary condition.
 We can continue this process to find $v_i$, which is the solution of

  \[
 \begin{cases}
 -\Delta v= (2^*-1) u_k^{2^*-2}K(y) v_{i-1},& \text{in}\; B_{2R}(0),\\
 v=0,& \text{on}\; \partial B_{2R}(0).
 \end{cases}
 \]
 and satisfies

 \[
 \begin{split}
 0\le  v_i (y)=& \int_{B_{2R}(0)}G(z, y) (2^*-1) u_k^{2^*-2} K(y)v_{i-1}\\
 \le &C \int_{B_{2R}(0)}\frac1{|y-z|^{N-2}}\frac1{|z-x|^{N-2(i-1)}}\,dz
 \le \frac{C}{|y-x|^{N-2i}}.
 \end{split}
 \]

 Let $i$ be large so that $v_i\in L^\infty(B_{2R}(0))$.  Define

 \[
 v=\sum_{l=1}^i v_l,
 \]
and $w= G_k(y, x)- \eta v^*$, where $\eta(y)=\eta(|y|)\in C^\infty_0(B_{2R}(0))$,$\eta=1$ in $B_{\frac32 R}(0)$  and $0\le \eta\le 1$. We then have

\begin{equation}\label{1-1-6-21}
L_k w = f,
\end{equation}
where $f\in L^\infty\cap H_s$ and $f=0$ in $\mathbb R^N\setminus B_{2R}(0)$.  By Theorem~\ref{th1-29-3},  \eqref{1-1-6-21}
 has a solution $w\in H_s\cap D^{1, 2}(\mathbb R^N)$.

It remains to prove that $|w(y)|\le \frac{C}{|y|^{N-2}}$ as $|y|\to +\infty$.

First, we claim that $|w|\le C |f|_{L^\infty(\mathbb R^N)}$.
In fact, suppose that there are $f_m\in L^\infty\cap H_s$, $w_m$ satisfying \eqref{1-1-6-21}, with $|f|_{L^\infty(\mathbb R^N)}\to 0$
and $|w_m|_{L^\infty(\mathbb R^N)}=1$.  Then, $w_m\to w$ in $C^1_{loc}(\mathbb R^N)$, which satisfies $L_k w=0$. Hence $w=0$. On the other hand,
we have

\begin{equation}\label{2-1-6-21}
w_m(y) =C_N \int_{\mathbb R^N}\frac1{|z-y|^{N-2} }(2^*-1) K(y) u_k^{2^*-2} w_m\,dz+  C_N \int_{\mathbb R^N}\frac1{|z-y|^{N-2} }
f_m
\,dz.
\end{equation}
Thus,  $|w_m(y)|\le \frac{C}{|y|^2}$ as $|y|\to +\infty$. This is a contradiction to $|w_m|_{L^\infty(\mathbb R^N)}=1$.

Now, $w$ is bounded, we obtain from \eqref{2-1-6-21} that  $|w(y)|\le \frac{C}{(1+|y|)^2}$. Then,

\[
|w(y)|\le C\int_{\mathbb R^N}\frac1{|z-y|^{N-2} } u_k^{2^*-2}\frac{C}{(1+|y|)^2} +\frac{C}{(1+|y|)^{N-2}}\le \frac{C}{(1+|y|)^4}.
\]
We can repeat this process to prove $|w(y)|\le \frac{C}{|y|^{N-2}}$ as $|y|\to +\infty$.
 \end{proof}

\bigskip
\noindent\textbf{Acknowledgements} \,\,\,Y. Guo was supported by NNSF of China (No. 11771235). M. Musso was supported by EPSRC research grant EP/T008458/1. S. Peng was supported by NNSF of China (No. 11571130, No. 11831009). S. Yan was supported by NNSF of China (No. 11629101).

\end{document}